\documentclass[12pt]{amsart}

\usepackage{amssymb}

\usepackage{enumerate}

\makeatletter
\@namedef{subjclassname@2010}{%
  \textup{2010} Mathematics Subject Classification}
\makeatother

\newtheorem{thm}{Theorem}[section]

\newtheorem{lem}[thm]{Lemma}
\newtheorem{prop}[thm]{Proposition}

\theoremstyle{definition}
\newtheorem{defin}[thm]{Definition}
\newtheorem{rem}[thm]{Remark}
\newtheorem{exa}[thm]{Example}

\newtheorem*{xrem}{Remark}

\numberwithin{equation}{section}

\begin{document}

\baselineskip=17pt



\title[Partial integrability of almost complex structures]{Partial integrability of almost complex structures on Thurston manifolds}

\author[O. Mushkarov]{Oleg Mushkarov}
\address{Institute of Mathematics and Informatics\\ Bulgarian Academy of Sciences\\
Acad. G.Bonchev Str. Bl. 8\\ 1113 Sofia\\ Bulgaria, and South-West
University, 2700 Blagoevgrad\\ Bulgaria}

\email{muskarov@math.bas.bg}

\author[Ch. Yankov]{Christian Yankov}
\address{Department of Mathematical Sciences \\
            Eastern Connecticut State University \\
            Willimantic, CT 06226, USA}
\email{yankovc@easternct.edu}

\date{}

\begin{abstract}
We prove that any left-invariant symplectic almost complex structure
on a Thurston manifold which is compatible with its canonical
left-invariant Riemannian metric has holomorphic type $1$.
\end{abstract}

\subjclass[2010]{Primary 32Q60; Secondary 22E30, 53D05}

\keywords{almost complex manifold, symplectic structure, holomorphic
function, holomorphic type, generalized Thurston manifold}

\maketitle

\section{Introduction}
\label{intro}
Let $M$ be a smooth $2n$-dimensional almost complex manifold with
almost complex structure $J$.  A smooth complex-valued function $f$
on $(M,J)$ is said to be holomorphic, if
 $\overline{\partial}f=0$, i.e. $df$ is a $(1,0)$-form with respect
to $J$ \cite{He,Mus1}. It is well-known that $J$ is integrable (i.e.
comes from a complex structure), if and only if in a neighborhood of
every point of $M$ there exist $n$ functionally independent
holomorphic functions. By the celebrated Newlander-Nirenberg theorem
\cite{NN} this is equivalent to the vanishing of the Nijenhuis
tensor of $J$. In contrast to complex manifolds, there exist almost
complex manifolds which do not admit even local holomorphic
functions except constants. A typical example is the 6-sphere $S^6$
with the almost complex structure defined by the Cayley numbers, and
more generally, every isotropy irreducible homogeneous almost
complex space with non-integrable almost complex structure
\cite{He}. Another class of examples is given by the compact
orientable hypersurfaces of $\mathbb{R}^7$ with the Calabi almost
complex structures \cite{Ca}. We refer the reader to \cite{Mus1,HK}
for some general criteria for local existence of $k$ $(0\leq k\leq
n)$ functionally independent holomorphic functions on an arbitrary
almost complex manifold.

Given an almost complex manifold $(M,J)$ and a point $x\in M$ we
denote by $m(x)$ the maximal number of local holomorphic functions
that are independent at $x$. In general $m(x)$ is not a constant,
hence we introduce the following

\begin{defin} The almost complex structure $J$ is said to be
{\it partially integrable} if $m(x)=m=const$ for all $x\in M$. In
this case we say that $J$ has {\it holomorphic type} $m$.
\end{defin}
\noindent One may find various examples of partially integrable
almost complex manifolds in \cite{Mus1,DM}. In particular, any
left-invariant almost complex structure on a Lie group is partially
integrable.

In 1976 Thurston \cite{Thu}  constructed the first example of a
compact symplectic manifold $W$ which admits no K\"ahler structures.
It is well-known \cite{Ab} that $W$ is a $4$-dimensional nilmanifold
and in \cite {Mus2} one of the authors noticed that every
left-invariant almost complex structure on the corresponding
$2$-step nilpotent Lie group has nonzero holomorphic type. Motivated
by this example he introduced the class of real Lie groups  that
have this property (calling them Lie groups of type $T$), and
exhibited many examples of such groups \cite {Mus3}. In particular,
any Lie group of dimension $2n$ whose commutator ideal has dimension
less than $n$ is of type $T$. Recently the authors proved \cite{MY}
that this algebraic property characterizes the $2$-step nilpotent
Lie groups of type $T$.

In this note we study the higher dimensional generalizations
$W^{2n+2}$ of the Thurston example introduced by Cordero,
Fern\'andez and de L\'eon \cite{CoFeLe} (see also  \cite{CoFeGr}).
More precisely, in Theorem \ref{thm:main theorem} we describe all
left-invariant symplectic almost complex structures on $W^{2n+2}$
which are compatible with its canonical left-invariant Riemannian
metric, and prove that each of these structures has holomorphic
type $1$. This generalizes a result of Kim (\cite{Km}, Theorem
2.7) (see also (\cite{Mus3}, Proposition 3.2)) for the standard
symplectic structure on $W^{2n+2}$. Note that an almost complex
structure $J$ compatible with a Riemannian metric $g$ is called
{\it symplectic} if the K\"ahler $2$-form $F(X,Y)=g(JX,Y)$ is
closed.

\section{Lie groups of type $T$}
Let $G$ be a connected Lie group and let $\mathfrak g$ be the Lie
algebra of left-invariant vector fields on $G$, which we identify as
usual with the tangent space to $G$ at its unit element. We also
identify the left-invariant almost complex structures on $G$ with
the endomorphisms $J$ of $\mathfrak g$, such that $J^2  = -Id$. The
Nijenhuis tensor $N$ of $J$ is defined by
\begin{displaymath}
     N(X,Y) = [X,Y] + J[J X,Y] + J[X,J Y] - [J X,J Y]
\end{displaymath}
for all  $X,Y \in \mathfrak g.$ We denote by $\mathfrak l\mathfrak
n$ the Nijenhuis space of $J$ defined by
\begin{displaymath}
     \mathfrak l\mathfrak n = Span \{N(X,Y) , \ X,Y \in \mathfrak g\}.
\end{displaymath}
It is easy to see that $N(X,Y) = -N(Y,X)$ and $N(JX,Y) = -JN(X,Y)$.
Hence $\mathfrak l\mathfrak n$ is a $J$-invariant subspace of
$\mathfrak g$, and thus its real dimension is even. Note that the
almost complex structure $J$ is integrable if and only if $\mathfrak
l\mathfrak n = \{0\}$.

As we already mentioned in the introduction, any left-invariant
almost complex structure $J$ on a Lie group $G$ is partially
integrable. Note that the computation of its holomorphic type can be
reduced to a purely algebraic problem as follows.
\begin{defin}
\label{defin:IJ-subalgebra} A Lie subalgebra $\mathfrak h$ of
$\mathfrak g$ is said to be an $IJ$-subalgebra if:
\begin{enumerate}
    \item $\mathfrak h$ is $J$-invariant
    \item $\mathfrak l\mathfrak n \subseteq \mathfrak h$
    \item $[X,Y] + J[X,JY] \in \mathfrak h$, whenever $X \in \mathfrak h$ and $Y \in \mathfrak g$.
\end{enumerate}
\end{defin}
\noindent Denote by $d(\mathfrak g,J)$ the real dimension of the
intersection of all $IJ$-sub\-al\-ge\-bras of $\mathfrak g$. It
has been proved in \cite{Mus3} that the holomorphic type of $J$ is
equal to  $\displaystyle \frac{1}{2}(\dim \mathfrak g-d(\mathfrak
g,J))$. In particular, we have the following

\begin{prop}
\label{prop:IJ} If the Nijenhuis space $\mathfrak l\mathfrak n$ of
$J$ is an $IJ$-subalgebra of $\mathfrak g$, then the holomorphic
type of $J$ is equal to $\displaystyle \frac{1}{2}(\dim \mathfrak
g-\dim \mathfrak l\mathfrak n )$.
\end{prop}

\begin{proof} Note that if the Nijenhuis space $\mathfrak l\mathfrak n$ of $J$
is an $IJ$-subalgebra of  $\mathfrak g$ then by property (2) of
Definition \ref{defin:IJ-subalgebra} it follows that the
intersection of all $IJ$-subalgebras of $\mathfrak g$ is $\mathfrak
l\mathfrak n$. Hence $d(\mathfrak g,J)=\dim \mathfrak l\mathfrak n $
and the holomorphic type of $J$ is equal to $\displaystyle
\frac{1}{2}(\dim \mathfrak g-\dim \mathfrak l\mathfrak n )$.
\end{proof}

Let us recall \cite{Mus3} that a Lie group $G$ is said to be of type
$T$, if every left-invariant almost complex structure $J$ on $G$ has
non-zero holomorphic type. As a consequence of the above formula for
the holomorphic type we obtain the following algebraic condition for
a Lie group to be of type $T$.

\begin{prop}

\cite{Mus3} Every $2n$-dimensional Lie group whose Lie algebra has
commutator ideal of dimension less than $n$ is of type $T$.

\label{theorem:referred proposition}
\end{prop}

\begin{proof} The given condition implies that
$[\mathfrak g,\mathfrak g]+J[\mathfrak g,\mathfrak g]$ is an $IJ$-
subalgebra of $\mathfrak g$ whose dimension is less than or equal to
$2n-2$. Hence $d(\mathfrak g,J)\leq 2n-2$ and by Proposition
\ref{prop:IJ} the holomorphic type of $J$ is at least $1$.
\end{proof}

Recall that a Lie group $G$ with Lie algebra $\mathfrak g$ is said
to be \textit{2-step nilpotent}, if $[\mathfrak g ,[\mathfrak
g,\mathfrak g]]=0$. These Lie groups admit rich geometric structures
and provide interesting examples of complex and symplectic manifolds
\cite{Sal}. It turns out that for them the sufficient condition in
Proposition \ref{theorem:referred proposition} for a Lie group to be
of type $T$ is also necessary \cite{MY}.

\begin{exa} The Thurston example mentioned in the introduction can be
generalized to higher dimensions by using the so-called generalized
Heisenberg groups $H(q,p)$ \cite{CoFeGr}. They consist of all real
matrices of the form
\begin{displaymath}
\begin{pmatrix}
I_p & A & B \\
0 & I_q & C \\
0 & 0 & I_q \\
\end{pmatrix}
\end{displaymath}
where $I_p$ denotes the identity $p \times p$ matrix, $A$ and $B$
are arbitrary $p \times q$ matrices, and $C$ is a diagonal $q \times
q$ matrix. Note that $$\dim H(q,p)= 2qp + q.$$ Let $p_1, p_2, \dots,
p_n$ be non-negative integers and $q_1, q_2, \dots, q_n$ be positive
integers such that $q_1 + q_2 + \dots + q_n \equiv 0\pmod 2$.
Consider the group $$W(q_1,p_1,\dots,q_n,p_n) = H(q_1,p_1)
\times\dots\times H(q_n,p_n).$$ Let $\mathfrak w$ be the Lie algebra
of $W(q_1,p_1,\dots,q_n,p_n)$. Then $$\displaystyle \dim [\mathfrak
w, \mathfrak w] = \sum_{i=1}^n q_i p_i,$$ which implies $\dim
\mathfrak w \ge 2\dim [\mathfrak w, \mathfrak w] +2.$ Hence, by
Proposition \ref{theorem:referred proposition}, the group
$W(q_1,p_1,\dots,q_n,p_n)$ is of type $T$.
\end{exa}

\begin{rem} It has been proved in \cite{MY} that the condition $\displaystyle \dim
[\mathfrak g,\mathfrak g] \le n-1$ is also necessary for  $4$ and
$6$-dimensional nilpotent Lie groups to be of type $T$. This raises
the question whether the same is also true for nilpotent Lie groups
of higher dimensions. Note that for general Lie groups the answer to
this question is negative as shown by the following
\begin{exa} Let $G$ be the solvable Lie group whose Lie algebra $\mathfrak g$ has a
basis $(e_1, \dots, e_{2n})$ such that
$$\mathfrak g:[e_i, e_{2n}] = e_i, \ 1 \le i \le 2n-1.$$ It has been proved in
\cite{Mus3} (see also \cite{Sal}), that every left-invariant almost
complex structure on $G$ is integrable. In particular, $G$ is of
type $T$ but it does not satisfy the above condition since
$\displaystyle \dim [\mathfrak g,\mathfrak g] =2n-1>n-1.$ This shows
that the answer to the question above is negative for the class of
solvable Lie groups.

Let us note that J. Milnor \cite{Mi} has also considered the Lie
groups in this example (in odd dimensions too) and characterized
them by the property that every left-invariant Riemannian metric has
sectional curvature of constant sign.
\end{exa}
\end{rem}

\section{ Partial integrability on generalized Thurston manifolds}

Let $H^{2n+2}$ be the Lie group of matrices of the form

\begin{displaymath}
M=
\begin{pmatrix}
I_n & P & Q&0 \\
0&1& r &0\\
0& 0  & 1 &0 \\
0& 0  & 0 &e^{2\pi is} \\
\end{pmatrix}
\end{displaymath}
where $I_n$ is the unit $n\times n$ matrix, $P=(p_i)$ and $Q=(q_i)$
are $n\times 1$ real matrices, and $r,s$ are real numbers. Then
$H^{2n+2}$ is a connected and simply-connected 2-step nilpotent Lie
group of dimension $2n+2$. Denote by $W^{2n+2}$ the compact
nilmanifold $H^{2n+2}/\Gamma$, where $\Gamma$ is the subgroup of
$H^{2n+2}$ consisting of all matrices with integer entries. Since
the group $H^{2n+2}$ is not abelian it follows from the
Benson-Gordon theorem \cite{BG} that $W^{2n+2}$ does not admit
K\"ahler structures (see also \cite{CoFeLe,CoFeGr}). We will see
instead that it admits a large family of left-invariant symplectic
almost complex structures.
\par Denote by $(x_1, x_2, \dots, x_{2n+2})$ the global coordinates on $H^{2n+2}$ defined by
\begin{displaymath}
x_i(M)= p_i, \  x_{n+1}(M)=r, \ x_{n+1+i}(M)=q_{i}, \
x_{2n+2}(M)=s
\end{displaymath}
for $M \in H^{2n+2}$ and $1\le i \le n$. The following 1-forms on
$H^{2n+2}$ are left-invariant:
\begin{displaymath}
\alpha_i= dx_i, \ \alpha_{n+1}= dx_{n+1}, \ \alpha_{n+1+i} =
dx_{n+1+i} - x_idx_{n+1}, \ \alpha_{2n+2}=dx_{2n+2}.
\end{displaymath}
It is easy to check that
\begin{equation}
\label{eq:1} \quad d\alpha_i=0 \ , \  d\alpha_{n+1+i} =
\alpha_{n+1}\wedge \alpha_i \ , \  1\le i \le n+1.
\end{equation}
Let $X^i$ be the dual vector field of $\alpha_i$. Then
$$ X^i=\frac{\partial}{\partial x_i}, \
X^{n+1}=\frac{\partial}{\partial x_{n+1}}+ \sum_{i=1}^n x_i
\frac{\partial}{\partial x_{n+1+i}},$$
$$ X^{n+1+i}=
\frac{\partial}{\partial x_{n+1+i}}, \
X^{2n+2}=\frac{\partial}{\partial x_{2n+2}}, \ 1\le i \le n.$$
Denote by $\mathfrak h_{2n+2}$ the Lie algebra of $H^{2n+2}$. The
left invariant vector fields $X^1, \dots, X^{2n+2}$ form a basis of
$\mathfrak h_{2n+2}$ and the only nonzero Lie brackets are
\begin{equation}
\label{eq:2} [X^i, X^{n+1}] = X^{n+1+i}, \  1 \le i \le n.
\end{equation}

Denote by $g$ the left-invariant Riemannian metric on $H^{2n+2}$
defined by

$$ g = \sum_{i=1}^{2n+2} \alpha_i \otimes \alpha_i.$$
\noindent It defines a Riemannian metric on $W^{2n+2}$ which we
denote by the same symbol. Let $J$ be a left-invariant almost
complex structure on $H^{2n+2}$ compatible with $g$. Set

$$ JX^i = \sum_{j=1}^{2n+2} a_j^i X^j, \ 1 \le i \le 2n+2. $$

\noindent Then the $(2n+2) \times (2n+2)$ matrix $J = (a_j^i)$ is
orthogonal and antisymmetric. The K\"ahler form $F$ of the almost
Hermitian structure $(g, J)$ is given by
$$  F = \frac{1}{2} \sum_{i,j=1}^{2n+2} a^{i}_{j} \alpha_{i} \wedge \alpha_{j}.$$
The structure $(g,J)$ is {\it symplectic} if $dF=0$.

\begin{thm}
\label{thm:main theorem} Let $J$ be an almost complex structure on
$W^{2n+2}$ induced by a left-invariant almost complex structure on
the group $H^{2n+2}$ and compatible with the metric $g$. If the
almost Hermitian structure $(g,J)$ is symplectic then $J$ has
holomorphic type $1$.
\end{thm}
\begin{proof} We first describe all left-invariant almost complex
structures $J$ on $H^{2n+2}$ such that $(g,J)$ is a symplectic
structure.

\begin{lem}
\label{lemma:block-matrix type} The almost Hermitian structure
$(g,J)$ is symplectic if and only if the matrix of $J$ has the
block-matrix form
    \begin{equation}
\label{eq:3} \qquad J =
\begin{pmatrix}
O & A \\
-^{t}\!A & O\\
\end{pmatrix}
,
\end{equation}
where $A$ is an orthogonal $(n+1)\times (n+1)$ matrix of the form

    \begin{equation}
\label{eq:4} \qquad A =
\begin{pmatrix}
B  & \boldsymbol{\vdots} \\
\dots & .\\
\end{pmatrix}
,
\end{equation}
in which $B$ is a symmetric $n\times n$ matrix.
\end{lem}
\begin{proof}

Using the structure equations \ref{eq:1} and the fact that the
matrix $J = (a_j^i)$ is antisymmetric we get
$$  2dF =  \sum_{i,j=1}^{2n+2} a^{i}_{j} ( d\alpha_{i} \wedge \alpha_{j} - \alpha_{i} \wedge d\alpha_{j}) = $$
$$   =  \sum_{j=1}^{2n+2} \sum_{i=1}^{n} a^{n+i+1}_{j} \alpha_{i} \wedge \alpha_{j} \wedge \alpha_{n+1}
     + \sum_{i=1}^{2n+2} \sum_{j=1}^{n} a^{i}_{n+j+1} \alpha_{i} \wedge \alpha_{j} \wedge \alpha_{n+1} = $$
$$ =  2 \sum_{j=1}^{2n+2} \sum_{i=1}^{n} a^{n+i+1}_{j} \alpha_{i} \wedge \alpha_{j} \wedge \alpha_{n+1}.$$
Hence, $dF = 0$ if and only if $a_{n+i+1}^{j} = a_{n+j+1}^{i},  1
\le i, j \le n$,  and $a_{n+j+1}^{n+i+1} = 0 \ , 1 \le i \le n, 1 \le
j \le n+1.$  The matrix of $J$ is orthogonal and antisymmetric and
we see easily that it has the block-matrix form
    \begin{displaymath}
 \qquad J =
\begin{pmatrix}
C & A \\
-^{t}\!A & C\\
\end{pmatrix}
,
\end{displaymath}
where $C$ and $A$ are $(n+1)\times (n+1)$ matrices and $A$ has the
form  \ref{eq:4}.  The identity $J^2 = -I_{2n+2}$ implies
$$ C^2 - A \, ^{t}\!A = - I_{n+1}, \quad CA = O, \quad ^{t}\!A C = O, \quad ^{t}\!AA = I_{n+1}.$$
Hence $C =  A\, ^{t}\!A C = O$, which completes the proof.
\end{proof}

Now we are ready to prove the theorem

In view of Lemma \ref{lemma:block-matrix type}  we may assume that
the matrix of $J$ has the form \ref{eq:3}. Then
\begin{equation}
 \label{eq:5}
JX^{i} \in Span\{X^{n+2}, \dots, X^{2n+2}\}, \ 1 \le i\le n+1 ,
\end{equation}

\begin{equation}
\label{eq:6}
JX^{i} \in Span\{X^{1},  \dots, X^{n+1}\}, \  n+2\le
i\le 2n+2 .
\end{equation}
This, together with \ref{eq:2}, implies that
$$   N(X^{i}, X^{n+1}) = X^{n+i+1}, \  1 \le i \le n$$
and
$$   N(X^{i}, X^{j}) \in Span \{ X^{n+2},\dots, X^{2n+1},  JX^{n+2}, \dots, JX^{2n+1} \}$$
for $1 \le i, j \le 2n+2. $
Hence, the Nijenhuis space of $J$ is given by
$$\mathfrak{l}\mathfrak{n}= Span \{ X^{n+2}, \dots, X^{2n+1},  JX^{n+2}, \dots, JX^{2n+1} \}$$
 and it follows from \ref{eq:5} and \ref{eq:6} that $\dim \mathfrak{l}\mathfrak{n} = 2n$.  On the other
 hand, using \ref{eq:2}, \ref{eq:5}, and \ref{eq:6}, one can easily check that $\mathfrak{l}\mathfrak{n}$ satisfies property (3) of
 Definition \ref{defin:IJ-subalgebra}, while properties (1) and (2) are
 obvious. Thus $\mathfrak{l}\mathfrak{n}$
 is an $IJ$-subalgebra
 of $\mathfrak{h}_{2n+2}$ and using Proposition \ref{prop:IJ} we conclude  that the almost complex structure $J$ has holomorphic
type 1.
\end{proof}

\begin{xrem}
Theorem \ref{thm:main theorem} generalizes a result of Kim
(\cite{Km}, Theorem 2.7) (see also (\cite{Mus3}, Proposition 3.2))
that the symplectic almost complex structure on $W^{2n+2}$ defined
by Lemma \ref{lemma:block-matrix type} when $A$ is the identity
$(n+1)\times(n+1)$ matrix has holomorphic type $1$.
\end{xrem}

\subsection*{Acknowledgements}
The first author is partially supported by the National Science
Fund, Ministry of Education and Science of Bulgaria under contract
DFNI-I 02/14. The second author is partially supported by a 2015
ECSU-AAUP Faculty Research Grant.

\end{document}